\documentclass[10pt]{article}

\usepackage{tikz, float}
\usetikzlibrary{patterns}
\usepackage{amssymb}
\usepackage{url}
\usepackage{amscd}
\usepackage{amsthm}
\usepackage{amsmath}
\usepackage{graphicx}

\newtheorem {Theorem}  {Theorem}
\newtheorem {Corollary} {Corollary}

\newtheorem {Problem} {Problem}

\begin{document}
\baselineskip = 15pt
\bibliographystyle{plain}

\title{Translational Aperiodic Sets of $7$ Polyominoes}
\date{}
\author{Chao Yang\\ 
              School of Mathematics and Statistics\\
              Guangdong University of Foreign Studies, Guangzhou, 510006, China\\
              sokoban2007@163.com, yangchao@gdufs.edu.cn\\
              \\
        Zhujun Zhang\\
              Big Data Center of Fengxian District, Shanghai, 201499, China\\
              zhangzhujun1988@163.com
              }
\maketitle

\begin{abstract}
Recently, two extraordinary results on aperiodic monotiles have been obtained in two different settings. One is a family of aperiodic monotiles in the plane discovered by Smith, Myers, Kaplan and Goodman-Strauss in 2023, where rotation is allowed, breaking the 50-year-old record (aperiodic sets of two tiles found by Roger Penrose in the 1970s) on the minimum size of aperiodic sets in the plane. The other is the existence of an aperiodic monotile in the translational tiling of $\mathbb{Z}^n$ for some huge dimension $n$ proved by Greenfeld and Tao. This disproves the long-standing periodic tiling conjecture. However, it is known that there is no aperiodic monotile for translational tiling of the plane. The smallest size of known aperiodic sets for translational tilings of the plane is $8$, which was discovered more than $30$ years ago by Ammann. In this paper, we prove that translational tiling of the plane with a set of $7$ polyominoes is undecidable. As a consequence of the undecidability, we have constructed a family of aperiodic sets of size $7$ for the translational tiling of the plane. This breaks the 30-year-old record of Ammann.
\end{abstract}


\section{Introduction}\label{sec_intro}

Aperiodicity and undecidability are two closely related phenomena in tiling. In general, under the same conditions, undecidability implies aperiodicity. But the reverse is not necessarily true. The study of aperiodicity and undecidability dates back to the problem of tiling the plane with Wang tiles, introduced by Hao Wang \cite{wang61}. A \textit{Wang tile} is a unit square with each edge assigned a color. Given a finite set of Wang tiles (see Figure \ref{fig_w3} for an example), Wang considered the problem of tiling the entire plane with translated copies from the set, under the conditions that the tiles must be edge-to-edge and the color of common edges of any two adjacent Wang tiles must be the same. This is known as \textit{Wang's domino problem}.


\begin{figure}[ht]
\begin{center}
\begin{tikzpicture}[scale=0.5]

\draw [ fill=green!50] (0,0)--(2,2)--(2,0)--(0,0);
\draw [ fill=red!50] (0,0)--(2,-2)--(2,0)--(0,0);
\draw [ fill=red!50] (4,0)--(2,2)--(2,0)--(4,0);
\draw [ fill=yellow!50] (4,0)--(2,-2)--(2,0)--(4,0);

\draw [ fill=yellow!50] (6+0,0)--(6+2,2)--(6+2,0)--(6+0,0);
\draw [ fill=blue!50] (6+0,0)--(6+2,-2)--(6+2,0)--(6+0,0);
\draw [ fill=blue!50] (6+4,0)--(6+2,2)--(6+2,0)--(6+4,0);
\draw [ fill=red!50] (6+4,0)--(6+2,-2)--(6+2,0)--(6+4,0);

\draw [ fill=red!50] (12+0,0)--(12+2,2)--(12+2,0)--(12+0,0);
\draw [ fill=yellow!50] (12+0,0)--(12+2,-2)--(12+2,0)--(12+0,0);
\draw [ fill=yellow!50] (12+4,0)--(12+2,2)--(12+2,0)--(12+4,0);
\draw [ fill=green!50] (12+4,0)--(12+2,-2)--(12+2,0)--(12+4,0);

\end{tikzpicture}
\end{center}
\caption{A set of $3$ Wang tiles} \label{fig_w3}
\end{figure}
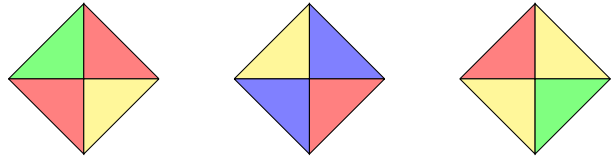

Wang's domino problem was shown to be undecidable by Berger \cite{b66}. As a crucial part of the proof of undecidability, Berger found a set of Wang tiles that tiles the plane but only tiles the plane non-periodically. Such a set is called \textit{aperiodic}. By combining the facts of the existence of an aperiodic set of Wang tiles and the ability to simulate the Turing machine with Wang tiles, Berger managed to show that Wang's domino problem is undecidable.

\begin{Theorem}[\cite{b66}]\label{thm_berger}
    Wang's domino problem is undecidable.
\end{Theorem}

Since Berger's proof of the undecidability of Wang's domino problem, aperiodicity and undecidability of tiling problems in different settings received extensive study. For a more comprehensive introduction to aperiodic tiling sets and their applications in quasicrystal, we refer to the books \cite{bg13, gs16}. We mention two remarkable recent results on aperiodic tilings in two different settings here. For tiling of the plane where rotation is allowed, Penrose first found an aperiodic set of two tiles \cite{p79}. That was the record holder for the smallest aperiodic sets in the setting of general tiling (i.e., rotation is allowed) of the plane for more than 50 years until a family of aperiodic monotiles was discovered by Smith, Myers, Kaplan, and Goodman-Strauss \cite{smith23a,smith23b} in 2023. Another remarkable result is for the translational tiling of $\mathbb{Z}^n$.  Greenfeld and Tao disprove the periodic tiling conjecture \cite{gs16, lw96, s74} by showing the existence of an aperiodic monotile in the translational tiling of $\mathbb{Z}^n$ in some extremely large dimension $n$ \cite{gt24a}. 

This paper focuses on the aperiodicity and undecidability of the following problem, which is a specific case of translational tiling of $\mathbb{Z}^n$.

\begin{Problem}[Translational Tiling Problem with $k$ Polyominoes] \label{pro_main}
Let $k$ be a fixed positive integer. Given a set $S$ of $k$ polyominoes, is there an algorithm to decide whether the entire plane can be tiled by translated copies of tiles in $S$?
\end{Problem}

If a polyomino can tile the plane, then it can always tile the plane periodically. As a result, the translational tiling problem with a single polyomino (i.e., $k$=1) is decidable \cite{bn91,b20, gt21, w15}. On the other hand, Ollinger initiated the study of the undecidability of the translational tiling problem with a fixed number of polyominoes in 2009, and he showed that the problem is undecidable for $k=11$ \cite{o09}. Since the size of the smallest aperiodic set of polyominoes for translational tiling of plane is 8 of Ammann's set \cite{ags92} from the 1990s (in fact, according to \cite{s04}, the discovery of Ammann appeared even earlier in unpublished private communication with Martin Gardner in 1976), Ollinger suggested investigating the undecidability of the translational tiling problem with $k$ polyominoes for $k=8,9,10$. Yang and Zhang confirmed that the translational tiling problems with $k$ polyominoes are indeed undecidable for $k=8,9,10$ in a series of works \cite{yang23, yang23b, yz24}.

The main contribution of this paper is to show that the translational tiling problem with $7$ polyominoes is also undecidable (Theorem \ref{thm_main}).

\begin{Theorem}[Undecidability with Seven Polyominoes]\label{thm_main}
    Translational tiling of the plane with a set of $7$ polyominoes is undecidable.
\end{Theorem}

As a corollary of Theorem \ref{thm_main}, there exist aperiodic sets of $7$ polyominoes. This is the first improvement in the size of the smallest aperiodic sets for translational tiling of the plane for more than 30 years.

\begin{Corollary}[Aperiodicity with Seven Polyominoes]\label{cor_main}
    There exist sets of $7$ polyominoes such that each set can tile the plane with translated copies, but only non-periodically.
\end{Corollary}

Like many other undecidable results on tiling problems \cite{dl24,gt23, gt24b,jr12,o09,yang23,yang23b,yz24,yz24a}, our main result (Theorem \ref{thm_main}) is proved by reduction from Wang's domino problem. The proof of Theorem \ref{thm_main} not only incorporates techniques we have developed in our previous work to show the undecidability of translational tiling problems in $3$-dimensional and $4$-dimensional spaces \cite{yz24b,yz24c,yz24d}, but also makes use of novel techniques that are introduced in the current paper. One of the key novel techniques is an \textbf{interlacing method} for simulating the Wang tiles in a single polyomino. This is significantly different from all previous reduction methods where distinct simulated Wang tiles are separated from each other in distinct parts of a polyomino.

The remainder of the paper is organized as follows. In Section \ref{sec_proof}, we prove Theorem \ref{thm_main} and Corollary \ref{cor_main}. Section \ref{sec_conclu} concludes with a few remarks on future work.

\section{Proof of the Main Results}\label{sec_proof}

Theorem \ref{thm_main} will be proved by reduction from Wang's domino problem in this section. For each instance of Wang's domino problem (i.e., a set of Wang tiles), we will construct a set of $7$ polyominoes that can tile the plane if and only if the corresponding set of Wang tiles can tile the plane.

\subsection{Building Blocks}
We begin with introducing a few building blocks for the construction of the more complex polyominoes. A \textit{polyomino} is the union of a finite number of unit squares gluing together edge-to-edge. A $10\times 10$ polyomino is called a \textit{functional square}. A functional square is a building block. By adding bumps or dents to functional squares, we obtain more building blocks.
 
The first two pairs of building blocks, $L$ and $l$, and $R$ and $r$, are illustrated in Figure \ref{fig_lr}. They are used for the binary encoding of the colors of Wang tiles. In all figures of this subsection (Figures \ref{fig_lr} to \ref{fig_b}), a gray square represents a unit square, and a white square means empty in that location. The building blocks $l$ and $r$ are functional squares with some vacant areas inside. The shape of the vacant areas is identical for the building blocks $l$ and $r$, and the only difference is the position of the vacant areas. As their names indicate, the vacant area of building block $l$ is on the left half of the functional square, while the vacant area is on the right half for the building block $r$. The building blocks $l$ and $r$ are disconnected by themselves. However, they will be attached to other building blocks to form a bigger and connected polyomino in the next subsection. Building blocks $L$ and $R$ can be viewed as the complement (with respect to a $10\times 10$ functional square) of building blocks $l$ and $r$, respectively. In other words, building blocks $l$ and $L$ (or $r$ and $R$) can be put together to form a $10\times 10$ functional square perfectly. The difference of building blocks $L$ and $R$ only makes sense when they are attached to other building blocks. When they are used independently as one of the $7$ polyominoes that will be defined in the next subsection, they are identical and referred to as the \textit{tiny filler} or $T$-filler.

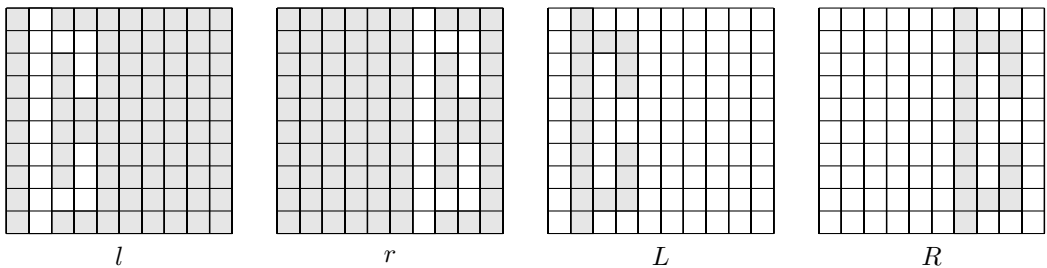
\begin{figure}[H]
\begin{center}
\begin{tikzpicture}[scale=0.3]

\foreach \x in {0}
\foreach \y in {0} 
{
\draw [fill=gray!20] (\x+0,10+\y)--(\x+10,10+\y)--(\x+10,0+\y)--(\x+0,0+\y)--(\x+0,10+\y);
}

\foreach \x in {0}
\foreach \y in {0} 
{
\draw [fill=white!20] (\x+1,\y)--(\x+1,10+\y)--(\x+2,10+\y)--(\x+2,9+\y)--(\x+4,9+\y)--(\x+4,6+\y)--(\x+3,6+\y)--(\x+3,8+\y)--(\x+2,8+\y)--(\x+2,2+\y)--(\x+3,2+\y)--(\x+3,4+\y)--(\x+4,4+\y)--(\x+4,1+\y)--(\x+2,1+\y)--(\x+2,\y)--(\x+1,\y);
}

\foreach \x in {12}
\foreach \y in {0} 
{
\draw [fill=gray!20] (\x+0,10+\y)--(\x+10,10+\y)--(\x+10,0+\y)--(\x+0,0+\y)--(\x+0,10+\y);
}

\foreach \x in {17}
\foreach \y in {0} 
{
\draw [fill=white!20] (\x+1,\y)--(\x+1,10+\y)--(\x+2,10+\y)--(\x+2,9+\y)--(\x+4,9+\y)--(\x+4,6+\y)--(\x+3,6+\y)--(\x+3,8+\y)--(\x+2,8+\y)--(\x+2,2+\y)--(\x+3,2+\y)--(\x+3,4+\y)--(\x+4,4+\y)--(\x+4,1+\y)--(\x+2,1+\y)--(\x+2,\y)--(\x+1,\y);
}

\foreach \x in {24}
\foreach \y in {0} 
{
\draw [fill=gray!20] (\x+1,\y)--(\x+1,10+\y)--(\x+2,10+\y)--(\x+2,9+\y)--(\x+4,9+\y)--(\x+4,6+\y)--(\x+3,6+\y)--(\x+3,8+\y)--(\x+2,8+\y)--(\x+2,2+\y)--(\x+3,2+\y)--(\x+3,4+\y)--(\x+4,4+\y)--(\x+4,1+\y)--(\x+2,1+\y)--(\x+2,\y)--(\x+1,\y);
}

\foreach \x in {41}
\foreach \y in {0} 
{
\draw [fill=gray!20] (\x+1,\y)--(\x+1,10+\y)--(\x+2,10+\y)--(\x+2,9+\y)--(\x+4,9+\y)--(\x+4,6+\y)--(\x+3,6+\y)--(\x+3,8+\y)--(\x+2,8+\y)--(\x+2,2+\y)--(\x+3,2+\y)--(\x+3,4+\y)--(\x+4,4+\y)--(\x+4,1+\y)--(\x+2,1+\y)--(\x+2,\y)--(\x+1,\y);
}

\foreach \x in {0,12,24,36}
\foreach \y in {0,...,10} 
{ 
\draw  (\x+0,0+\y)--(\x+10,0+\y);
\draw  (\x,0)--(\x,10);
\draw  (\x+10,0)--(\x+10,10);
\draw  (\x+9,0)--(\x+9,10);
\draw  (\x+8,0)--(\x+8,10);
\draw  (\x+7,0)--(\x+7,10);
\draw  (\x+6,0)--(\x+6,10);
\draw  (\x+5,0)--(\x+5,10);
\draw  (\x+4,0)--(\x+4,10);
\draw  (\x+3,0)--(\x+3,10);
\draw  (\x+2,0)--(\x+2,10);
\draw  (\x+1,0)--(\x+1,10);
}

\node at (5,-1) {$l$};  \node at (17,-1) {$r$};  \node at (29,-1) {$L$};  \node at (41,-1) {$R$};

\end{tikzpicture}
\end{center}
\caption{Building blocks $l$, $r$, $L$ and $R$.}\label{fig_lr}
\end{figure}

The next two pairs of building blocks, $Y^+$ and $y^+$, and $Y^-$ and $y^-$, are illustrated in Figure \ref{fig_y}. These building blocks are used for ensure the connection of polyominoes in the direction of the vertical axis of the plane ($y$-axis). The building blocks $Y^+$ and $Y^-$ are functional squares with a bump on the north and south sides, respectively. The building blocks $y^+$ and $y^-$ are functional squares with a dent that matches the bump of the building blocks $Y^+$ and $Y^-$, respectively.

The last three pairs of building blocks, $X$ and $x$, $A$ and $a$, and $B$ and $b$, are illustrated in Figure \ref{fig_x}, Figure \ref{fig_a} and Figure \ref{fig_b}, respectively. They are used for ensuring the connection of the polyominoes in the horizontal direction of the plane (direction of the $x$-axis). Each pair of these three building blocks has a unique shape of bump or dent that can only be matched with each other. In other words, bumps or dents of different pairs cannot be matched.


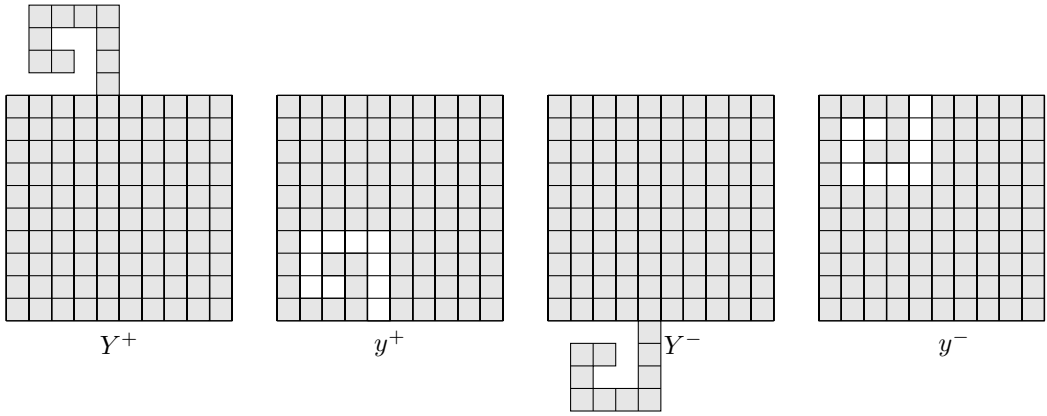
\begin{figure}[H]
\begin{center}
\begin{tikzpicture}[scale=0.3]

\foreach \x in {0}
\foreach \y in {0} 
{
\draw [fill=gray!20] (\x+0,10+\y)--(\x+10,10+\y)--(\x+10,0+\y)--(\x+0,0+\y)--(\x+0,10+\y);
}

\foreach \x in {0}
\foreach \y in {0} 
{
\draw [fill=gray!20] (\x+4,10+\y)--(\x+4,13+\y)--(\x+2,13+\y)--(\x+2,12+\y)--(\x+3,12+\y)--(\x+3,11+\y)--(\x+1,11+\y)--(\x+1,14+\y)--(\x+5,14+\y)--(\x+5,10+\y)--(\x+4,10+\y);
\draw (\x+4,11+\y)--(\x+5,11+\y);
\draw (\x+4,12+\y)--(\x+5,12+\y);
\draw (\x+4,13+\y)--(\x+5,13+\y);
\draw (\x+1,12+\y)--(\x+2,12+\y);
\draw (\x+1,13+\y)--(\x+2,13+\y);

\draw (\x+2,11+\y)--(\x+2,12+\y);
\draw (\x+2,13+\y)--(\x+2,14+\y);
\draw (\x+3,13+\y)--(\x+3,14+\y);
\draw (\x+4,13+\y)--(\x+4,14+\y);
}

\foreach \x in {12}
\foreach \y in {0} 
{
\draw [fill=gray!20] (\x+0,10+\y)--(\x+10,10+\y)--(\x+10,0+\y)--(\x+0,0+\y)--(\x+0,10+\y);
}

\foreach \x in {12}
\foreach \y in {-10} 
{
\draw [fill=white!20] (\x+4,10+\y)--(\x+4,13+\y)--(\x+2,13+\y)--(\x+2,12+\y)--(\x+3,12+\y)--(\x+3,11+\y)--(\x+1,11+\y)--(\x+1,14+\y)--(\x+5,14+\y)--(\x+5,10+\y)--(\x+4,10+\y);
}

\foreach \x in {24}
\foreach \y in {0} 
{
\draw [fill=gray!20] (\x+0,10+\y)--(\x+10,10+\y)--(\x+10,0+\y)--(\x+0,0+\y)--(\x+0,10+\y);
}

\foreach \x in {24}
\foreach \y in {-10} 
{
\draw [fill=gray!20] (\x+4,10+\y)--(\x+4,7+\y)--(\x+2,7+\y)--(\x+2,8+\y)--(\x+3,8+\y)--(\x+3,9+\y)--(\x+1,9+\y)--(\x+1,6+\y)--(\x+5,6+\y)--(\x+5,10+\y)--(\x+4,10+\y);
\draw (\x+4,9+\y)--(\x+5,9+\y);
\draw (\x+4,8+\y)--(\x+5,8+\y);
\draw (\x+4,7+\y)--(\x+5,7+\y);
\draw (\x+1,8+\y)--(\x+2,8+\y);
\draw (\x+1,7+\y)--(\x+2,7+\y);

\draw (\x+2,9+\y)--(\x+2,8+\y);
\draw (\x+2,7+\y)--(\x+2,6+\y);
\draw (\x+3,7+\y)--(\x+3,6+\y);
\draw (\x+4,7+\y)--(\x+4,6+\y);
}

\foreach \x in {36}
\foreach \y in {0} 
{

\draw [fill=gray!20] (\x+0,10+\y)--(\x+10,10+\y)--(\x+10,0+\y)--(\x+0,0+\y)--(\x+0,10+\y);
}

\foreach \x in {36}
\foreach \y in {0} 
{
\draw [fill=white!20] (\x+4,10+\y)--(\x+4,7+\y)--(\x+2,7+\y)--(\x+2,8+\y)--(\x+3,8+\y)--(\x+3,9+\y)--(\x+1,9+\y)--(\x+1,6+\y)--(\x+5,6+\y)--(\x+5,10+\y)--(\x+4,10+\y);
}

\foreach \x in {0,12,24,36}
\foreach \y in {0,...,10} 
{ 
\draw  (\x+0,0+\y)--(\x+10,0+\y);
\draw  (\x,0)--(\x,10);
\draw  (\x+10,0)--(\x+10,10);
\draw  (\x+9,0)--(\x+9,10);
\draw  (\x+8,0)--(\x+8,10);
\draw  (\x+7,0)--(\x+7,10);
\draw  (\x+6,0)--(\x+6,10);
\draw  (\x+5,0)--(\x+5,10);
\draw  (\x+4,0)--(\x+4,10);
\draw  (\x+3,0)--(\x+3,10);
\draw  (\x+2,0)--(\x+2,10);
\draw  (\x+1,0)--(\x+1,10);
}

\node at (5,-1) {$Y^+$};  \node at (17,-1) {$y^+$};  \node at (30,-1) {$Y^-$};  \node at (42,-1) {$y^-$};

\end{tikzpicture}
\end{center}
\caption{Building blocks $Y^+$, $y^+$, $Y^-$ and $y^-$.}\label{fig_y}
\end{figure}


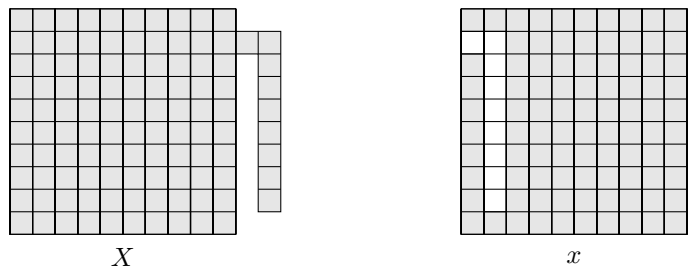
\begin{figure}[H]
\begin{center}
\begin{tikzpicture}[scale=0.3]

\foreach \x in {0}
\foreach \y in {0} 
{
\draw [fill=gray!20] (\x+0,10+\y)--(\x+10,10+\y)--(\x+10,0+\y)--(\x+0,0+\y)--(\x+0,10+\y);
}

\foreach \x in {0}
\foreach \y in {0} 
{
\draw [fill=gray!20] (\x+10,9+\y)--(\x+12,9+\y)--(\x+12,1+\y)--(\x+11,1+\y)--(\x+11,8+\y)--(\x+10,8+\y)--(\x+10,9+\y);

\draw (\x+11,8+\y)--(\x+11,9+\y);

\draw (\x+11,2+\y)--(\x+12,2+\y);
\draw (\x+11,3+\y)--(\x+12,3+\y);
\draw (\x+11,4+\y)--(\x+12,4+\y);
\draw (\x+11,5+\y)--(\x+12,5+\y);
\draw (\x+11,6+\y)--(\x+12,6+\y);
\draw (\x+11,7+\y)--(\x+12,7+\y);
\draw (\x+11,8+\y)--(\x+12,8+\y);

}

\foreach \x in {20}
\foreach \y in {0} 
{
\draw [fill=gray!20] (\x+0,10+\y)--(\x+10,10+\y)--(\x+10,0+\y)--(\x+0,0+\y)--(\x+0,10+\y);
}

\foreach \x in {10}
\foreach \y in {0} 
{
\draw [fill=white!20] (\x+10,9+\y)--(\x+12,9+\y)--(\x+12,1+\y)--(\x+11,1+\y)--(\x+11,8+\y)--(\x+10,8+\y)--(\x+10,9+\y);
}

\foreach \x in {0,20}
\foreach \y in {0,...,10} 
{ 
\draw  (\x+0,0+\y)--(\x+10,0+\y);
\draw  (\x,0)--(\x,10);
\draw  (\x+10,0)--(\x+10,10);
\draw  (\x+9,0)--(\x+9,10);
\draw  (\x+8,0)--(\x+8,10);
\draw  (\x+7,0)--(\x+7,10);
\draw  (\x+6,0)--(\x+6,10);
\draw  (\x+5,0)--(\x+5,10);
\draw  (\x+4,0)--(\x+4,10);
\draw  (\x+3,0)--(\x+3,10);
\draw  (\x+2,0)--(\x+2,10);
\draw  (\x+1,0)--(\x+1,10);
}

\node at (5,-1) {$X$};  \node at (25,-1) {$x$};

\end{tikzpicture}
\end{center}
\caption{Building blocks $X$ and $x$.}\label{fig_x}
\end{figure}


\begin{figure}[H]
\begin{center}
\begin{tikzpicture}[scale=0.3]

\foreach \x in {0}
\foreach \y in {0} 
{
\draw [fill=gray!20] (\x+0,10+\y)--(\x+10,10+\y)--(\x+10,0+\y)--(\x+0,0+\y)--(\x+0,10+\y);
}

\foreach \x in {0}
\foreach \y in {0} 
{
\draw [fill=gray!20] (\x+10,9+\y)--(\x+14,9+\y)--(\x+14,6+\y)--(\x+12,6+\y)--(\x+12,1+\y)--(\x+11,1+\y)--(\x+11,7+\y)--(\x+13,7+\y)--(\x+13,8+\y)--(\x+10,8+\y)--(\x+10,9+\y);

\draw (\x+11,8+\y)--(\x+11,9+\y);
\draw (\x+12,8+\y)--(\x+12,9+\y);
\draw (\x+13,8+\y)--(\x+13,9+\y);
\draw (\x+12,6+\y)--(\x+12,7+\y);
\draw (\x+13,6+\y)--(\x+13,7+\y);

\draw (\x+11,2+\y)--(\x+12,2+\y);
\draw (\x+11,3+\y)--(\x+12,3+\y);
\draw (\x+11,4+\y)--(\x+12,4+\y);
\draw (\x+11,5+\y)--(\x+12,5+\y);
\draw (\x+11,6+\y)--(\x+12,6+\y);
\draw (\x+13,7+\y)--(\x+14,7+\y);
\draw (\x+13,8+\y)--(\x+14,8+\y);

}

\foreach \x in {20}
\foreach \y in {0} 
{
\draw [fill=gray!20] (\x+0,10+\y)--(\x+10,10+\y)--(\x+10,0+\y)--(\x+0,0+\y)--(\x+0,10+\y);
}

\foreach \x in {10}
\foreach \y in {0} 
{
\draw [fill=white!20] (\x+10,9+\y)--(\x+14,9+\y)--(\x+14,6+\y)--(\x+12,6+\y)--(\x+12,1+\y)--(\x+11,1+\y)--(\x+11,7+\y)--(\x+13,7+\y)--(\x+13,8+\y)--(\x+10,8+\y)--(\x+10,9+\y);
}

\foreach \x in {0,20}
\foreach \y in {0,...,10} 
{ 
\draw  (\x+0,0+\y)--(\x+10,0+\y);
\draw  (\x,0)--(\x,10);
\draw  (\x+10,0)--(\x+10,10);
\draw  (\x+9,0)--(\x+9,10);
\draw  (\x+8,0)--(\x+8,10);
\draw  (\x+7,0)--(\x+7,10);
\draw  (\x+6,0)--(\x+6,10);
\draw  (\x+5,0)--(\x+5,10);
\draw  (\x+4,0)--(\x+4,10);
\draw  (\x+3,0)--(\x+3,10);
\draw  (\x+2,0)--(\x+2,10);
\draw  (\x+1,0)--(\x+1,10);
}

\node at (5,-1) {$A$};  \node at (25,-1) {$a$};

\end{tikzpicture}
\end{center}
\caption{Building blocks $A$ and $a$.}\label{fig_a}
\end{figure}
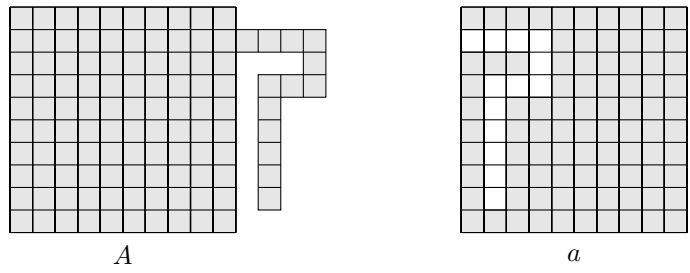


\begin{figure}[H]
\begin{center}
\begin{tikzpicture}[scale=0.3]

\foreach \x in {0}
\foreach \y in {0} 
{
\draw [fill=gray!20] (\x+0,10+\y)--(\x+10,10+\y)--(\x+10,0+\y)--(\x+0,0+\y)--(\x+0,10+\y);
}

\foreach \x in {0}
\foreach \y in {0} 
{
\draw [fill=gray!20] (\x+10,9+\y)--(\x+14,9+\y)--(\x+14,2+\y)--(\x+12,2+\y)--(\x+12,1+\y)--(\x+11,1+\y)--(\x+11,7+\y)--(\x+12,7+\y)--(\x+12,3+\y)--(\x+13,3+\y)--(\x+13,8+\y)--(\x+10,8+\y)--(\x+10,9+\y);

\draw (\x+11,8+\y)--(\x+11,9+\y);
\draw (\x+12,8+\y)--(\x+12,9+\y);
\draw (\x+13,8+\y)--(\x+13,9+\y);
\draw (\x+12,2+\y)--(\x+12,3+\y);
\draw (\x+13,2+\y)--(\x+13,3+\y);

\draw (\x+11,2+\y)--(\x+12,2+\y);
\draw (\x+11,3+\y)--(\x+12,3+\y);
\draw (\x+11,4+\y)--(\x+12,4+\y);
\draw (\x+11,5+\y)--(\x+12,5+\y);
\draw (\x+11,6+\y)--(\x+12,6+\y);
\draw (\x+13,7+\y)--(\x+14,7+\y);
\draw (\x+13,8+\y)--(\x+14,8+\y);
\draw (\x+13,3+\y)--(\x+14,3+\y);
\draw (\x+13,4+\y)--(\x+14,4+\y);
\draw (\x+13,5+\y)--(\x+14,5+\y);
\draw (\x+13,6+\y)--(\x+14,6+\y);
}

\foreach \x in {20}
\foreach \y in {0} 
{
\draw [fill=gray!20] (\x+0,10+\y)--(\x+10,10+\y)--(\x+10,0+\y)--(\x+0,0+\y)--(\x+0,10+\y);
}

\foreach \x in {10}
\foreach \y in {0} 
{
\draw [fill=white!20] (\x+10,9+\y)--(\x+14,9+\y)--(\x+14,2+\y)--(\x+12,2+\y)--(\x+12,1+\y)--(\x+11,1+\y)--(\x+11,7+\y)--(\x+12,7+\y)--(\x+12,3+\y)--(\x+13,3+\y)--(\x+13,8+\y)--(\x+10,8+\y)--(\x+10,9+\y);
}

\foreach \x in {0,20}
\foreach \y in {0,...,10} 
{ 
\draw  (\x+0,0+\y)--(\x+10,0+\y);
\draw  (\x,0)--(\x,10);
\draw  (\x+10,0)--(\x+10,10);
\draw  (\x+9,0)--(\x+9,10);
\draw  (\x+8,0)--(\x+8,10);
\draw  (\x+7,0)--(\x+7,10);
\draw  (\x+6,0)--(\x+6,10);
\draw  (\x+5,0)--(\x+5,10);
\draw  (\x+4,0)--(\x+4,10);
\draw  (\x+3,0)--(\x+3,10);
\draw  (\x+2,0)--(\x+2,10);
\draw  (\x+1,0)--(\x+1,10);
}

\node at (5,-1) {$B$};  \node at (25,-1) {$b$};

\end{tikzpicture}
\end{center}
\caption{Building blocks $B$ and $b$.}\label{fig_b}
\end{figure}
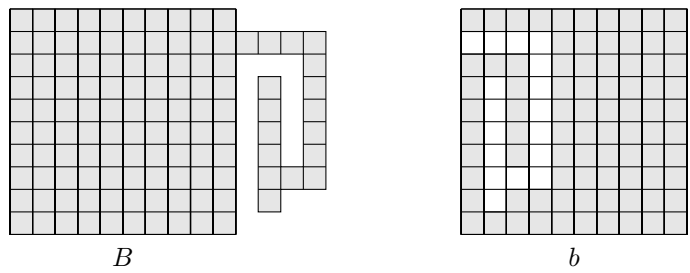

\subsection{The Set of Seven Polyominoes}
 
With the building blocks introduced in the previous subsection, we are now ready to describe the construction of the set of $7$ polyominoes. We use the set of $3$ Wang tiles in Figure \ref{fig_w3} as an example to explain our construction method. The method can be applied to an arbitrary set of Wang tiles in the same way.

The first polyomino that simulates the set of $3$ Wang tiles of Figure \ref{fig_w3} is the \textit{encoder} as illustrated at the bottom of Figure \ref{fig_encoder}. Each square represents a building block, namely a $10\times 10$ functional square with or without bumps/dents. Note that this differs from the figures in the previous subsection where each square represents a unit square. To be distinguished from the figures of the previous subsection, Figure \ref{fig_encoder} is referred to as \textit{level}-$2$ diagram. In Figure \ref{fig_encoder}, the gray squares without labels are normal functional squares. The gray squares with labels ($Y^+$, $Y^-$, $a$ or $B$) are building blocks introduced in the previous subsection. The colored squares (red, green, blue or yellow) are building blocks of either $l$ or $r$. The colored square are not labeled in the complete encoder at the bottom of Figure \ref{fig_encoder}, but labeled in the decomposed encoders on the top three rows as we will describe soon in the next paragraph. So the encoder is a polyomino of size $30\times 3$ (counted by building blocks rather than unit squares).

As we have mentioned in Section \ref{sec_intro}, the set of $3$ Wang tiles of Figure \ref{fig_w3} are simulated in this encoder polyomino in an interlacing manner. To see each of the simulated Wang tiles clearly, we decompose the encoder into three as illustrated on the top three rows of Figure \ref{fig_encoder}, with the Wang tile being simulated on the right of each row for comparison. For each of the decomposed encoders, only the building blocks that are relevant to the simulated Wang tiles are depicted. So each of the Wang tiles is simulated by exactly $8$ building blocks, and the labels of the relevant building blocks together with their locations within the complete encoder are depicted. Because there are $4$ different edge colors in the set of Wang tiles of Figure \ref{fig_w3}, so $2$ binary bits suffice to encoder all the colors. The colors red, green, yellow and blue are encoded by a sequence of two building blocks $ll$, $lr$, $rl$ and $rr$, respectively.


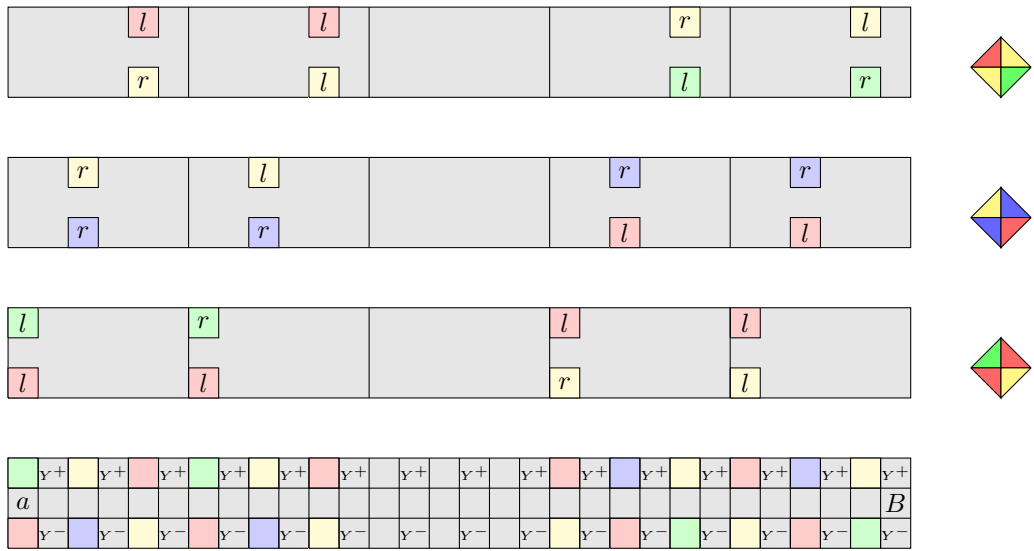
\begin{figure}[H]
\begin{center}
\begin{tikzpicture}[scale=0.4,pattern1/.style={draw=red,pattern color=gray!70, pattern=north east lines}]

\draw [ fill=gray!20] (0,0)--(30,0)--(30,3)--(0,3)--(0,0);
\draw [ fill=gray!20] (0,5)--(30,5)--(30,8)--(0,8)--(0,5);
\draw [ fill=gray!20] (0,10)--(30,10)--(30,13)--(0,13)--(0,10);
\draw [ fill=gray!20] (0,15)--(30,15)--(30,18)--(0,18)--(0,15);

\foreach \x in {0,6}
\foreach \y in {0,5}
{
\draw [ fill=red!20] (\x+0,\y+0)--(\x+0,\y+1)--(\x+1,\y+1)--(\x+1,\y+0)--(\x+0,\y+0);
}
\foreach \x in {18,24}
\foreach \y in {0,5}
{
\draw [ fill=yellow!20] (\x+0,\y+0)--(\x+0,\y+1)--(\x+1,\y+1)--(\x+1,\y+0)--(\x+0,\y+0);
}
\foreach \x in {18,24}
\foreach \y in {2,7}
{
\draw [ fill=red!20] (\x+0,\y+0)--(\x+0,\y+1)--(\x+1,\y+1)--(\x+1,\y+0)--(\x+0,\y+0);
}
\foreach \x in {0,6}
\foreach \y in {2,7}
{
\draw [ fill=green!20] (\x+0,\y+0)--(\x+0,\y+1)--(\x+1,\y+1)--(\x+1,\y+0)--(\x+0,\y+0);
}

\node at (0+0.5,5+0.5) {$l$}; 
\node at (6+0.5,5+0.5) {$l$};

\node at (18+0.5,5+0.5) {$r$}; 
\node at (24+0.5,5+0.5) {$l$};

\node at (0+0.5,7+0.5) {$l$}; 
\node at (6+0.5,7+0.5) {$r$};

\node at (18+0.5,7+0.5) {$l$}; 
\node at (24+0.5,7+0.5) {$l$};

\foreach \x in {2,8}
\foreach \y in {0,10}
{
\draw [ fill=blue!20] (\x+0,\y+0)--(\x+0,\y+1)--(\x+1,\y+1)--(\x+1,\y+0)--(\x+0,\y+0);
}
\foreach \x in {20,26}
\foreach \y in {0,10}
{
\draw [ fill=red!20] (\x+0,\y+0)--(\x+0,\y+1)--(\x+1,\y+1)--(\x+1,\y+0)--(\x+0,\y+0);
}
\foreach \x in {20,26}
\foreach \y in {2,12}
{
\draw [ fill=blue!20] (\x+0,\y+0)--(\x+0,\y+1)--(\x+1,\y+1)--(\x+1,\y+0)--(\x+0,\y+0);
}
\foreach \x in {2,8}
\foreach \y in {2,12}
{
\draw [ fill=yellow!20] (\x+0,\y+0)--(\x+0,\y+1)--(\x+1,\y+1)--(\x+1,\y+0)--(\x+0,\y+0);
}

\node at (0+2.5,8+2.5) {$r$}; 
\node at (6+2.5,8+2.5) {$r$};

\node at (18+2.5,8+2.5) {$l$}; 
\node at (24+2.5,8+2.5) {$l$};

\node at (0+2.5,10+2.5) {$r$}; 
\node at (6+2.5,10+2.5) {$l$};

\node at (18+2.5,10+2.5) {$r$}; 
\node at (24+2.5,10+2.5) {$r$};

\foreach \x in {4,10}
\foreach \y in {0,15}
{
\draw [ fill=yellow!20] (\x+0,\y+0)--(\x+0,\y+1)--(\x+1,\y+1)--(\x+1,\y+0)--(\x+0,\y+0);
}
\foreach \x in {22,28}
\foreach \y in {0,15}
{
\draw [ fill=green!20] (\x+0,\y+0)--(\x+0,\y+1)--(\x+1,\y+1)--(\x+1,\y+0)--(\x+0,\y+0);
}
\foreach \x in {22,28}
\foreach \y in {2,17}
{
\draw [ fill=yellow!20] (\x+0,\y+0)--(\x+0,\y+1)--(\x+1,\y+1)--(\x+1,\y+0)--(\x+0,\y+0);
}
\foreach \x in {4,10}
\foreach \y in {2,17}
{
\draw [ fill=red!20] (\x+0,\y+0)--(\x+0,\y+1)--(\x+1,\y+1)--(\x+1,\y+0)--(\x+0,\y+0);
}
\node at (0+4.5,15+0.5) {$r$}; 
\node at (6+4.5,15+0.5) {$l$};

\node at (18+4.5,15+0.5) {$l$}; 
\node at (24+4.5,15+0.5) {$r$};

\node at (0+4.5,17+0.5) {$l$}; 
\node at (6+4.5,17+0.5) {$l$};

\node at (18+4.5,17+0.5) {$r$}; 
\node at (24+4.5,17+0.5) {$l$};

\foreach \y in {1,2}
{
\draw (0,\y)--(30,\y);
}
\foreach \x in {1,...,29}
{
\draw (\x,0)--(\x,3);
}
\foreach \x in {6,12,18,24}
\foreach \y in {5,10,15}
{
\draw (\x,\y+0)--(\x,\y+3);
}

\foreach \x in {1,3,5,7,9,11, 13,15,17, 19,21,23,25,27,29}
\foreach \y in {0}
{
\node at (\x+0.5,\y+0.5) {\tiny $Y^-$};
}
\foreach \x in {1,3,5,7,9,11, 13,15,17, 19,21,23,25,27,29}
\foreach \y in {2}
{
\node at (\x+0.5,\y+0.5) {\tiny $Y^+$};
}

\foreach \x in {32}
\foreach \y in {5}
{
\draw [ fill=green!60] (\x+0,1+\y)--(\x+1,1+\y)--(\x+1, +2+\y)--(\x+0, +1+\y);
\draw [ fill=red!60] (\x+2,1+\y)--(\x+1,1+\y)--(\x+1, +2+\y)--(\x+2, +1+\y);
\draw [ fill=red!60] (\x+0,1+\y)--(\x+1,1+\y)--(\x+1, +0+\y)--(\x+0, +1+\y);
\draw [ fill=yellow!60] (\x+2,1+\y)--(\x+1,1+\y)--(\x+1, +0+\y)--(\x+2, +1+\y);
}

\foreach \x in {32}
\foreach \y in {10}
{
\draw [ fill=yellow!60] (\x+0,1+\y)--(\x+1,1+\y)--(\x+1, +2+\y)--(\x+0, +1+\y);
\draw [ fill=blue!60] (\x+2,1+\y)--(\x+1,1+\y)--(\x+1, +2+\y)--(\x+2, +1+\y);
\draw [ fill=blue!60] (\x+0,1+\y)--(\x+1,1+\y)--(\x+1, +0+\y)--(\x+0, +1+\y);
\draw [ fill=red!60] (\x+2,1+\y)--(\x+1,1+\y)--(\x+1, +0+\y)--(\x+2, +1+\y);
}

\foreach \x in {32}
\foreach \y in {15}
{
\draw [ fill=red!60] (\x+0,1+\y)--(\x+1,1+\y)--(\x+1, +2+\y)--(\x+0, +1+\y);
\draw [ fill=yellow!60] (\x+2,1+\y)--(\x+1,1+\y)--(\x+1, +2+\y)--(\x+2, +1+\y);
\draw [ fill=yellow!60] (\x+0,1+\y)--(\x+1,1+\y)--(\x+1, +0+\y)--(\x+0, +1+\y);
\draw [ fill=green!60] (\x+2,1+\y)--(\x+1,1+\y)--(\x+1, +0+\y)--(\x+2, +1+\y);
}

\foreach \x in {0}
\foreach \y in {1}
{
\node at (\x+0.5,\y+0.5) {$a$};
}
\foreach \x in {29}
\foreach \y in {1}
{
\node at (\x+0.5,\y+0.5) {$B$};
}

\end{tikzpicture}
\end{center}
\caption{Level-2 diagram of the encoder.}\label{fig_encoder}
\end{figure}

To see it in another way, the encoder consists of $5$ segments as illustrated in the top three rows of Figure \ref{fig_encoder}. Each segment is of size $6\times 3$ (of building blocks). The top sides and bottom sides of the two segments on the left simulate the northwest sides and southwest sides of the Wang tiles, respectively. The two segments on the right simulate the northeast and southeast sides of the Wang tiles. Note that the segment in the middle has no building blocks $l$ or $r$. Therefore, the middle segment does not contribute in encoding the Wang tiles but plays a role in the overall structure of the tiling. So the middle segment is called the \textit{structural segment}, and all the other segments are called \textit{encoding segments}.

Note that another feature of the encoder is that on the north side (south side, resp.), the building blocks appear alternatively between a building block $Y^+$ ($Y^-$, resp.) and building blocks $l$/$r$ or a functional square. In the middle of the west side (east side, resp.) of the encoder is a building block $a$ ($B$, resp.). The building blocks $Y^+$, $Y^-$, $a$ and $B$ of the encoder also serve for structural purposes.

In general, for a set of $n$ Wang tiles with $m$ different colors (let $t=\lceil \log_2 m\rceil$), the corresponding encoder is constructed in the same way. The general encoder is a polyomino of size $2n(2t+1)\times 3$ (of building blocks). In other words, the encoder consists of $2t+1$ segments: one structural segment in the middle, $t$ encoding segments on the left and $t$ encoding segments on the right. Each segment is of size $2n\times 3$ (of building blocks).

The other $6$ polyominoes: two \textit{linkers} ($L$-linker and $R$-linker), two \textit{bigger fillers} ($A$-filler and $B$-filler), a \textit{connector}, and a \textit{tiny filler} ($T$-filler), are illustrated in Figure \ref{fig_poly_6}. By the same convention as Figure \ref{fig_encoder}, a gray square without labels represents a functional square and a gray square with a label represents a special building block of that label introduced in the previous subsection.


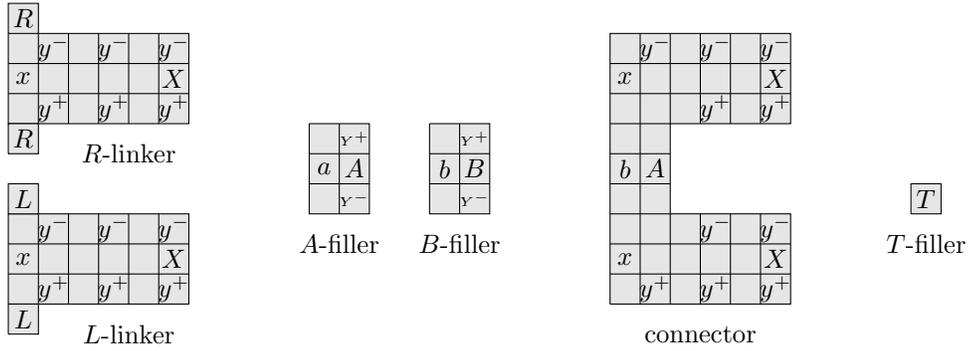
\begin{figure}[H]
\begin{center}
\begin{tikzpicture}[scale=0.4,pattern1/.style={draw=red,pattern color=gray!70, pattern=north east lines}]

\draw [ fill=gray!20] (0,0)--(6,0)--(6,3)--(0,3)--(0,0);
\draw [ fill=gray!20] (0,9)--(6,9)--(6,6)--(0,6)--(0,9);

\draw [ fill=gray!20] (20,0)--(26,0)--(26,3)--(22,3)--(22,6)--(26,6)--(26,9)--(20,9)--(20,0);

\draw [ fill=gray!20] (10,6)--(12,6)--(12,3)--(10,3)--(10,6);
\draw [ fill=gray!20] (14,6)--(16,6)--(16,3)--(14,3)--(14,6);

\foreach \x in {0}
\foreach \y in {-1,3,5,9}
{
\draw [ fill=gray!20] (\x+0,\y+0)--(\x+0,\y+1)--(\x+1,\y+1)--(\x+1,\y+0)--(\x+0,\y+0);
}

\foreach \y in {1,2,7,8}
{
\draw (0,\y)--(6,\y);
}
\foreach \x in {1,...,5}
{
\draw (\x,0)--(\x,3);\draw (\x,6)--(\x,9);
}

\foreach \y in {1,2,7,8}
{
\draw (20,\y)--(26,\y);
}
\foreach \y in {3,4,5,6}
{
\draw (20,\y)--(22,\y);
}
\foreach \x in {21}
{
\draw (\x,0)--(\x,9);
}
\foreach \x in {22,...,25}
{
\draw (\x,0)--(\x,3); 
\draw (\x,6)--(\x,9);
}

\foreach \y in {4,5}
{
\draw (10,\y)--(12,\y);
\draw (14,\y)--(16,\y);
}
\foreach \x in {11,15}
{
\draw (\x,6)--(\x,3);
}

\foreach \x in {30}
\foreach \y in {3}
{
\draw [ fill=gray!20] (\x+0,\y+0)--(\x+0,\y+1)--(\x+1,\y+1)--(\x+1,\y+0)--(\x+0,\y+0);
\node at (\x+0.5,\y+0.5) {$T$};
}

\foreach \x in {1,3,5}
\foreach \y in {6}
{
\node at (\x+0.5,\y+0.5) {$y^+$};
}
\foreach \x in {1,3,5}
\foreach \y in {8}
{
\node at (\x+0.5,\y+0.5) {$y^-$};
}
\foreach \x in {1,3,5,  21,23,25}
\foreach \y in {0}
{
\node at (\x+0.5,\y+0.5) {$y^+$};
}
\foreach \x in {1,3,5,  23,25}
\foreach \y in {2}
{
\node at (\x+0.5,\y+0.5) {$y^-$};
}
\foreach \x in { 23,25}
\foreach \y in {6}
{
\node at (\x+0.5,\y+0.5) {$y^+$};
}
\foreach \x in { 21,23,25}
\foreach \y in {8}
{
\node at (\x+0.5,\y+0.5) {$y^-$};
}

\foreach \x in {0}
\foreach \y in {-1,3}
{
\node at (\x+0.5,\y+0.5) {$L$};
}
\foreach \x in {0}
\foreach \y in {5,9}
{
\node at (\x+0.5,\y+0.5) {$R$};
}

\foreach \x in {0}
\foreach \y in {7}
{
\node at (\x+0.5,\y+0.5) {$x$};
}
\foreach \x in {5}
\foreach \y in {7}
{
\node at (\x+0.5,\y+0.5) {$X$};
}
\foreach \x in {0,20}
\foreach \y in {1}
{
\node at (\x+0.5,\y+0.5) {$x$};
}
\foreach \x in {5,25}
\foreach \y in {1}
{
\node at (\x+0.5,\y+0.5) {$X$};
}
\foreach \x in {20}
\foreach \y in {7}
{
\node at (\x+0.5,\y+0.5) {$x$};
}
\foreach \x in {25}
\foreach \y in {7}
{
\node at (\x+0.5,\y+0.5) {$X$};
}

\foreach \x in {11,15}
\foreach \y in {3}
{
\node at (\x+0.5,\y+0.5) {{\tiny $Y^-$}};
}
\foreach \x in {11,15}
\foreach \y in {5}
{
\node at (\x+0.5,\y+0.5) {{\tiny $Y^+$}};
}

\foreach \x in {10}
\foreach \y in {4}
{
\node at (\x+0.5,\y+0.5) {$a$};
}
\foreach \x in {11}
\foreach \y in {4}
{
\node at (\x+0.5,\y+0.5) {$A$};
}
\foreach \x in {14}
\foreach \y in {4}
{
\node at (\x+0.5,\y+0.5) {$b$};
}
\foreach \x in {15}
\foreach \y in {4}
{
\node at (\x+0.5,\y+0.5) {$B$};
}

\foreach \x in {20}
\foreach \y in {4}
{
\node at (\x+0.5,\y+0.5) {$b$};
}
\foreach \x in {21}
\foreach \y in {4}
{
\node at (\x+0.5,\y+0.5) {$A$};
}

\node at (4,-1) {$L$-linker};
\node at (4,5) {$R$-linker};

\node at (11,2) {$A$-filler};
\node at (15,2) {$B$-filler};

\node at (23,-1) {connector};

\node at (30.5,2) {$T$-filler};

\end{tikzpicture}
\end{center}
\caption{Level-2 diagram of two linkers, two bigger filler, one connector, one tiny filler.}\label{fig_poly_6}
\end{figure}

Both linkers, the $L$-\textit{linker} and $R$-\textit{linker}, consist of a main body of size $6\times 3$ (of building blocks), and two building blocks attached to the north and south sides of the main body, respectively. The main bodies of the two linkers are the same, with building blocks $y^-$, $y^+$, $x$ and $X$ on the boundary. The only difference between $L$-linker and $R$-linker is the two building blocks attached to the main bodies. Two building blocks $R$ ($L$, resp.) are attached to the main body of $R$-linker ($L$-linker, resp.). The linkers are used to link the building blocks $l$ or $r$ of one encoder to that of another encoder, as the two building blocks $l$ and $r$ only appear in encoders among the set of $7$ polyominoes. Evidently, a building block $l$ ($r$, resp.) can only be linked to another building block $l$ ($r$, resp.).

The two bigger fillers, the $A$-\textit{filler} and $B$-\textit{filler} are of size $2\times 3$ (of building blocks). Both fillers have a building block $Y^+$ on the north and a building block $Y^-$ on the south. The $A$-linkers has building blocks $a$ and $A$, while the $B$-linkers has building blocks $b$ and $B$. They are used to fill the horizontal gaps (if any) between the encoders and connectors.

As we have seen previously, the encoders are linked by linkers. The linkers are in turn connected by the \textit{connectors} to form an infinite rigid structure as we will see later. Intuitively speaking, as we can see in the central right of Figure~\ref{fig_poly_6}, the connector is the union of two degenerated linkers (i.e., the main bodies of the linkers without the building blocks $L$ or $R$) and a mixed bigger filler (i.e., a mixture of $A$-filler and $B$-filler as it has a building block $b$ on the west side and a building block $A$ on the east side). When gluing together to form the connector, the two degenerated linkers and the mixed filler are aligned to the left.

The last one of our set of $7$ polyominoes, the tiny $T$-\textit{filler}, is just a building block $L$ (or $R$) as mentioned in the previous subsection.

Note that the construction of the two bigger filler and the tiny filler are the same for any set of Wang tiles. Like the encoder, the size of the two linkers and the connector change according to the set of Wang tiles. In general, for a set of $n$ Wang tiles with $m$ different colors (let $t=\lceil \log_2 m\rceil$), the main body of the linkers are of size $2n\times 3$ (of building blocks), which is the same as the size of a segment (encoding segments or structural segment) of the encoder. The size of the connector is determined by the linker and the bigger filler.

It is easy to check all the $7$ polyominoes constructed in this subsection are connected.

\subsection{Tiling Pattern and Proof of Theorem \ref{thm_main}}

\begin{proof}[Proof of Theorem \ref{thm_main}] Given a set of Wang tiles, we have constructed a set of $7$ polyominoes by the method in the previous subsection. To complete the proof, it suffices to show that the set of $7$ polyominoes can tile the plane if and only if the corresponding set of Wang tiles can tile the plane.

\begin{itemize}
\item (\textbf{The connector must be used.}) If the encoder is used, then the connector must be used. Because we can only put a connector to the south or north of the structural segment of an encoder in any tiling of the plane. If the $T$-filler is used, then the encoder must be used. Because the $T$-fillers cannot tile the plane by themselves, and the only building blocks $l$ and $r$ that can match the $T$-filler appear exclusively in the encoder. By the same reason, if the $L$-linker or $R$-linker is used, then the encoder must be used, as there are building blocks $L$ or $R$ (which are in fact identical to the $T$-filler) in the two linkers. 

If the $A$-filler is used, then the only other polyomino that can match the building block $A$ in the $A$-filler is the building block $a$. Note that the build block $a$ only appears in the encoder or the $A$-filler. If the encoder is used to match the building block $A$, then connector must in turn be used. If the encoder is not used, then we must form a horizontal array of $A$-filler, and the only tile that can be placed next to the north or south side of this array is the connector. By similar argument, if the $B$-filler is used, then the connector must be used.

In summary, we have shown that in all cases, the connector must be used to tile the entire plane.

\item (\textbf{Five segments away to the east or west of a connector, we must put another connector.}) In general, for the polyominoes constructed from a set of Wang tiles with $m$ colors (let $t=\lceil \log_2 m\rceil$), we must put another connector on the same horizontal row exactly $2t+1$ segments away to the east or west of any connector (see Figure \ref{fig_2conn}). Therefore, the distance between two horizontally neighboring connector is exactly the length of an encoder.

\begin{figure}[H]
\begin{center}
\begin{tikzpicture}[scale=0.15,pattern1/.style={draw=red,pattern color=gray!70, pattern=north east lines}]

\foreach \x in {12,48}
\foreach \y in {3}
{
\draw  [ fill=violet!20] (\x,\y)--(\x+6,\y)--(\x+6,\y+3)--(\x+2,\y+3)--(\x+2,\y+6)--(\x+6,\y+6)--(\x+6,\y+9)--(\x,\y+9)--(\x,\y);
}

\draw [<->,thick] (18,1)--(48,1);  
\node at (33,3) {$2t+1$ segments};
\draw (18,0)--(18,2); 
\draw (48,0)--(48,2);

\end{tikzpicture}
\end{center}
\caption{The distance between two neighboring connectors.}\label{fig_2conn}
\end{figure}
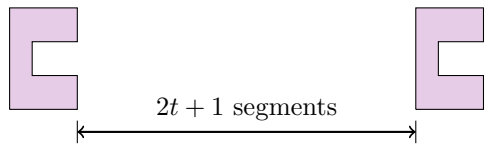


\begin{figure}[H]
\begin{center}
\begin{tikzpicture}[scale=0.15,pattern1/.style={draw=red,pattern color=gray!70, pattern=north east lines}]

\foreach \x in {12}
\foreach \y in {3}
{
\draw  [ fill=violet!20] (\x,\y)--(\x+6,\y)--(\x+6,\y+3)--(\x+2,\y+3)--(\x+2,\y+6)--(\x+6,\y+6)--(\x+6,\y+9)--(\x,\y+9)--(\x,\y);
}
\foreach \x in {60,66}
\foreach \y in {-3}
{
\draw  [ fill=violet!20] (\x,\y)--(\x+6,\y)--(\x+6,\y+3)--(\x+2,\y+3)--(\x+2,\y+6)--(\x+6,\y+6)--(\x+6,\y+9)--(\x,\y+9)--(\x,\y);
}
\filldraw[red] (64,1.5) circle (8pt);
\foreach \x in {14,16,18,20,22,24,26}
\foreach \y in {6}
{
\draw  [ fill=blue!20]  (\x,\y)--(\x+2,\y)--(\x+2,\y+3)--(\x,\y+3)--(\x,\y);
}
\foreach \x in {28}
\foreach \y in {6}
{
\draw  [color=black, fill=gray!20] (\x,\y)--(\x+30,\y)--(\x+30,\y+3)--(\x,\y+3)--(\x,\y);
}

\foreach \x in {58,60,62,64,66,68,70,72,74,76,78}
\foreach \y in {6}
{
\draw  [ fill=green!20]  (\x,\y)--(\x+2,\y)--(\x+2,\y+3)--(\x,\y+3)--(\x,\y);
}
\node at (82,7.5) {$\cdots$};

\end{tikzpicture}
\end{center}
\caption{No connectors to the east of a connector.}\label{fig_1conn}
\end{figure}
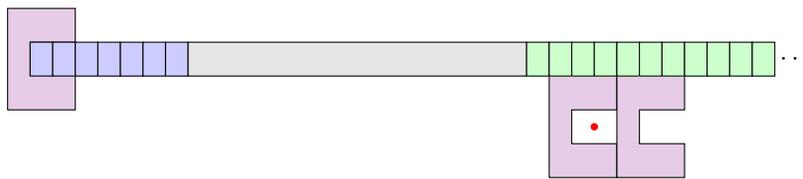

Suppose to the contrast that there is no connector exactly $2t+1$ segment away to the east of a connector, then there are two possibilities. (\textbf{Case I.}) If there are no connectors at all to the east of a connector, then there are only two methods to place tiles to the east of the building block $A$ of the connector. The first method is to put a one-way infinite array of $A$-fillers. The second method is to put in order a finite number of $A$-fillers, followed by a single encoder, and followed by a one-way infinite array $B$-fillers (see Figure \ref{fig_1conn}). Either method yields a contradiction immediately. For the second method, we can only place connectors to the south of the infinite array of $B$-fillers. But then there is no way to fill the holes between these connectors anymore (see the red solid dot of Figure \ref{fig_1conn}). For the first method, the tiling cannot continue further for the same reason.

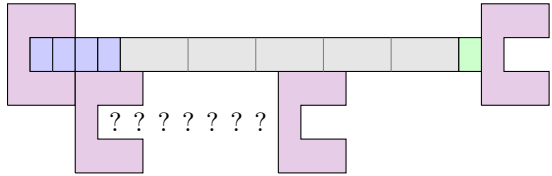
\begin{figure}[H]
\begin{center}
\begin{tikzpicture}[scale=0.15,pattern1/.style={draw=red,pattern color=gray!70, pattern=north east lines}]

\foreach \x in {12,54}
\foreach \y in {3}
{
\draw  [ fill=violet!20] (\x,\y)--(\x+6,\y)--(\x+6,\y+3)--(\x+2,\y+3)--(\x+2,\y+6)--(\x+6,\y+6)--(\x+6,\y+9)--(\x,\y+9)--(\x,\y);
}

\foreach \x in {18,36}
\foreach \y in {-3}
{
\draw  [ fill=violet!20] (\x,\y)--(\x+6,\y)--(\x+6,\y+3)--(\x+2,\y+3)--(\x+2,\y+6)--(\x+6,\y+6)--(\x+6,\y+9)--(\x,\y+9)--(\x,\y);
}
\node at (28,1.5) {? ? ? ? ? ? ?};
\foreach \x in {14,16,18,20}
\foreach \y in {6}
{
\draw  [ fill=blue!20]  (\x,\y)--(\x+2,\y)--(\x+2,\y+3)--(\x,\y+3)--(\x,\y);
}
\foreach \x in {22}
\foreach \y in {6}
{
\draw  [color=black, fill=gray!20] (\x,\y)--(\x+30,\y)--(\x+30,\y+3)--(\x,\y+3)--(\x,\y);
}
\foreach \x in {28,34,40,46}
{
\draw  [color=black!50] (\x,6)--(\x,9);
}

\foreach \x in {52}
\foreach \y in {6}
{
\draw  [ fill=green!20]  (\x,\y)--(\x+2,\y)--(\x+2,\y+3)--(\x,\y+3)--(\x,\y);
}

\end{tikzpicture}
\end{center}
\caption{Two connectors too far apart.}\label{fig_2conn_far}
\end{figure}

(\textbf{Case II.}) Now suppose there is indeed another connector to the east of a connector, but the distance between them is not exactly $2t+1$ segments. Note that to tile the plane, the distance of two horizontally neighboring connectors must be a multiple of segments. Otherwise, there is no way to fill the gap between the building block $X$ of the first connector and the building block $x$ of the second connector (because only linkers or connectors can be used to fill this gap). So we further divide into two subcases. If the two connectors are too close and the distance between them is $2t$ or less segments, then we get an immediate contradiction by the same red-solid-dot argument as in Figure \ref{fig_1conn} (or the question-mark-argument in Figure \ref{fig_2conn_far}). If the two connectors are too far apart and the distance between them is $2t+2$ or more segments, then the gap between the building block $A$ of the first connector and the building block $b$ of the second can only be filled in order (from left to right) a finite number (possibly zero) of $A$-fillers, followed by an encoder, and then followed by a finite number (possibly zero) of $B$-fillers (see Figure \ref{fig_2conn_far}). As the two connectors are too far apart, at least $n$ $A$-fillers and $B$-fillers in total ($n=3$ in this example) are exposed outside. So we must place a connector somewhere to the south of an exposed bigger filler (either $A$-filler or $B$-filler). We also have to place another connector to the south of the structural segment of the encoder. Hence we have a new pair of horizontally neighboring connectors that are too close to each other, and there is no way to fill the gap labeled by the question marks in Figure \ref{fig_2conn_far}.

So we have completed the proof of the fact that two horizontally neighboring connectors must be exactly $2t+1$ segments apart.

\item (\textbf{Three different configurations between two horizontally neighboring connectors.}) In general there are $n$ (where $n$ is the number of Wang tiles in a set) different configurations between two horizontally neighboring connectors. Figure \ref{fig_3config} illustrates all the three configurations for the example $n=3$. The different positions of the encoder simulate the flexibility of choosing one of the three Wang tiles.

No matter which configuration, we must place another two connectors to the north and south sides adjacent to the structural segment of the encoder (see the dark violet connectors in Figure \ref{fig_3config}). The newly added connectors must be placed exactly in the middle between the two initial connectors regardless of the configuration of the encoder. In other words, the distances between the newly added connectors (in dark violet) and the initial connector (in light violet) are $t$ segments ($t=2$ in this example). This is because the gaps between the building block $x$ of the dark connector and the building block $X$ of the light connector (or the gaps between the building block $X$ of the dark connector and the building block $x$ of the light connector) can only be filled with the linkers, which is always a multiple of segments.


\begin{figure}[H]
\begin{center}
\begin{tikzpicture}[scale=0.15,pattern1/.style={draw=red,pattern color=gray!70, pattern=north east lines}]

\foreach \x in {30}
\foreach \y in {-3, 9,21,33}
{
\draw  [ fill=violet!40] (\x,\y)--(\x+6,\y)--(\x+6,\y+3)--(\x+2,\y+3)--(\x+2,\y+6)--(\x+6,\y+6)--(\x+6,\y+9)--(\x,\y+9)--(\x,\y);
}

\foreach \x in {12,48}
\foreach \y in {3}
{
\draw  [ fill=violet!20] (\x,\y)--(\x+6,\y)--(\x+6,\y+3)--(\x+2,\y+3)--(\x+2,\y+6)--(\x+6,\y+6)--(\x+6,\y+9)--(\x,\y+9)--(\x,\y);
}

\foreach \x in {14}
\foreach \y in {6}
{
\draw  [color=black, fill=gray!20] (\x,\y)--(\x+30,\y)--(\x+30,\y+3)--(\x,\y+3)--(\x,\y);
}
\foreach \x in {20,26,32,38}
\foreach \y in {0}
{
\draw  [color=black!50] (\x,\y+6)--(\x,\y+9);
}

\foreach \x in {44,46}
\foreach \y in {6}
{
\draw  [ fill=green!20]  (\x,\y)--(\x+2,\y)--(\x+2,\y+3)--(\x,\y+3)--(\x,\y);
}

\foreach \x in {12,48}
\foreach \y in {15}
{
\draw  [ fill=violet!20] (\x,\y)--(\x+6,\y)--(\x+6,\y+3)--(\x+2,\y+3)--(\x+2,\y+6)--(\x+6,\y+6)--(\x+6,\y+9)--(\x,\y+9)--(\x,\y);
}

\foreach \x in {14}
\foreach \y in {18}
{
\draw  [ fill=blue!20]  (\x,\y)--(\x+2,\y)--(\x+2,\y+3)--(\x,\y+3)--(\x,\y);
}
\foreach \x in {16}
\foreach \y in {18}
{
\draw  [color=black, fill=gray!20] (\x,\y)--(\x+30,\y)--(\x+30,\y+3)--(\x,\y+3)--(\x,\y);
}
\foreach \x in {22,28,34,40}
\foreach \y in {12}
{
\draw  [color=black!50] (\x,\y+6)--(\x,\y+9);
}

\foreach \x in {46}
\foreach \y in {18}
{
\draw  [ fill=green!20]  (\x,\y)--(\x+2,\y)--(\x+2,\y+3)--(\x,\y+3)--(\x,\y);
}

\foreach \x in {12,48}
\foreach \y in {27}
{
\draw  [ fill=violet!20] (\x,\y)--(\x+6,\y)--(\x+6,\y+3)--(\x+2,\y+3)--(\x+2,\y+6)--(\x+6,\y+6)--(\x+6,\y+9)--(\x,\y+9)--(\x,\y);
}

\foreach \x in {14,16}
\foreach \y in {30}
{
\draw  [ fill=blue!20]  (\x,\y)--(\x+2,\y)--(\x+2,\y+3)--(\x,\y+3)--(\x,\y);
}
\foreach \x in {18}
\foreach \y in {30}
{
\draw  [color=black, fill=gray!20] (\x,\y)--(\x+30,\y)--(\x+30,\y+3)--(\x,\y+3)--(\x,\y);
}
\foreach \x in {24,30,36,42}
\foreach \y in {24}
{
\draw  [color=black!50] (\x,\y+6)--(\x,\y+9);
}

\draw [<->,thick] (18,1)--(30,1);  
\node at (24,3) {$t$ segments};
\draw (18,0)--(18,2); 
\draw (30,0)--(30,2);

\draw [<->,thick] (36,1)--(48,1);  
\node at (42,3) {$t$ segments};
\draw (36,0)--(36,2); 
\draw (48,0)--(48,2);

\end{tikzpicture}
\end{center}
\caption{Three different configurations between two connectors.}\label{fig_3config}
\end{figure}
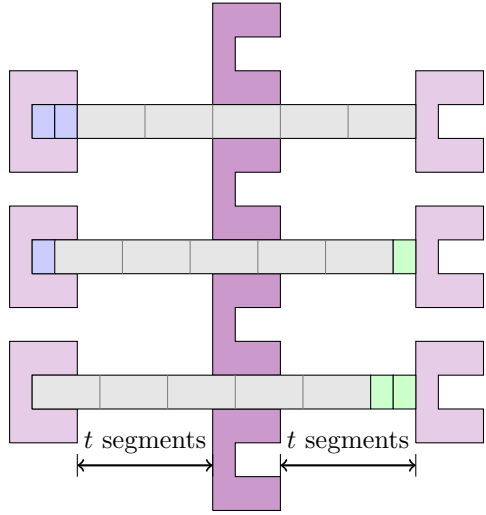

\item (\textbf{Forming the rigid tiling pattern.}) Now we can repeatedly apply the above properties about the positions of the connectors to get the overall tiling pattern as illustrated in Figure \ref{fig_pattern}.


\begin{figure}[H]
\begin{center}
\begin{tikzpicture}[scale=0.15,pattern1/.style={draw=red,pattern color=gray!70, pattern=north east lines}]


\foreach \x in {0,6,12,18,24}
\foreach \y in {0}
{
\draw  [color=black!50, fill=gray!20] (\x,\y)--(\x+6,\y)--(\x+6,\y+3)--(\x,\y+3)--(\x,\y);
}

\foreach \x in {32,38,44,50,56}
\foreach \y in {0}
{
\draw  [color=black!50, fill=gray!20] (\x,\y)--(\x+6,\y)--(\x+6,\y+3)--(\x,\y+3)--(\x,\y);
}

\foreach \x in {-18,-12,-6,0,6,16,22,28,34,40,  52,58,64,70,76}
\foreach \y in {6}
{
\draw  [color=black!50, fill=gray!20] (\x,\y)--(\x+6,\y)--(\x+6,\y+3)--(\x,\y+3)--(\x,\y);
}

\foreach \x in {-2,4,10,16,22}
\foreach \y in {12}
{
\draw  [color=black!50, fill=gray!20] (\x,\y)--(\x+6,\y)--(\x+6,\y+3)--(\x,\y+3)--(\x,\y);
}
\foreach \x in {36,42,48,54,60}
\foreach \y in {12}
{
\draw  [color=black!50, fill=gray!20] (\x,\y)--(\x+6,\y)--(\x+6,\y+3)--(\x,\y+3)--(\x,\y);
}
\foreach \x in {0,6,18,24,36,42,54,60}
\foreach \y in {-3,3,9,15}
{
\draw  [ fill=orange!20] (\x,\y-1)--(\x+1,\y-1)--(\x+1,\y)--(\x+6,\y)--(\x+6,\y+3)--(\x+1,\y+3)--(\x+1,\y+4)--(\x,\y+4)--(\x,\y-1);
}

\foreach \x in {-18,-12,72,78}
\foreach \y in {3,9}
{
\draw  [ fill=orange!20] (\x,\y-1)--(\x+1,\y-1)--(\x+1,\y)--(\x+6,\y)--(\x+6,\y+3)--(\x+1,\y+3)--(\x+1,\y+4)--(\x,\y+4)--(\x,\y-1);
}
\foreach \x in {-6,30,66}
\foreach \y in {-3,9}
{
\draw  [ fill=violet!20] (\x,\y)--(\x+6,\y)--(\x+6,\y+3)--(\x+2,\y+3)--(\x+2,\y+6)--(\x+6,\y+6)--(\x+6,\y+9)--(\x,\y+9)--(\x,\y);
}
\foreach \x in {-24,12,48,84}
\foreach \y in {3}
{
\draw  [ fill=violet!20] (\x,\y)--(\x+6,\y)--(\x+6,\y+3)--(\x+2,\y+3)--(\x+2,\y+6)--(\x+6,\y+6)--(\x+6,\y+9)--(\x,\y+9)--(\x,\y);
}
\foreach \x in {12,48}
\foreach \y in {-9,15}
{
\draw  [ fill=violet!20] (\x,\y)--(\x+6,\y)--(\x+6,\y+3)--(\x+2,\y+3)--(\x+2,\y+6)--(\x+6,\y+6)--(\x+6,\y+9)--(\x,\y+9)--(\x,\y);
}

\foreach \x in {-2,-4}
\foreach \y in {0}
{
\draw  [ fill=blue!20]  (\x,\y)--(\x+2,\y)--(\x+2,\y+3)--(\x,\y+3)--(\x,\y);
}
\foreach \x in {-4,32,34}
\foreach \y in {12}
{
\draw  [ fill=blue!20]  (\x,\y)--(\x+2,\y)--(\x+2,\y+3)--(\x,\y+3)--(\x,\y);
}
\foreach \x in {-20,-22,14,50}
\foreach \y in {6}
{
\draw  [ fill=blue!20]  (\x,\y)--(\x+2,\y)--(\x+2,\y+3)--(\x,\y+3)--(\x,\y);
}
\foreach \x in {62,64}
\foreach \y in {0}
{
\draw  [ fill=green!20]  (\x,\y)--(\x+2,\y)--(\x+2,\y+3)--(\x,\y+3)--(\x,\y);
}
\foreach \x in {46,82}
\foreach \y in {6}
{
\draw  [ fill=green!20]  (\x,\y)--(\x+2,\y)--(\x+2,\y+3)--(\x,\y+3)--(\x,\y);
}
\foreach \x in {28}
\foreach \y in {12}
{
\draw  [ fill=green!20]  (\x,\y)--(\x+2,\y)--(\x+2,\y+3)--(\x,\y+3)--(\x,\y);
}

\draw [thick] (80,-2)--(82,0)--(84,-2)--(82,-4)--(80,-2);
\draw [thick] (84,-2)--(86,0)--(88,-2)--(86,-4)--(84,-2);
\draw [thick] (80,-6)--(82,-8)--(84,-6)--(82,-4)--(80,-6);
\draw [thick] (84,-6)--(86,-8)--(88,-6)--(86,-4)--(84,-6);

\draw [thick] (80,-2)--(78,-4)--(80,-6); \draw [thick] (88,-2)--(90,-4)--(88,-6);

\draw [color=black!20] (80,-2)--(88,-2);
\draw [color=black!20] (78,-4)--(90,-4);
\draw [color=black!20] (80,-6)--(88,-6);

\draw [color=black!20] (82,0)--(82,-8);
\draw [color=black!20] (86,0)--(86,-8);

\draw [color=black!20] (80,-2)--(80,-6);
\draw [color=black!20] (84,-2)--(84,-6);
\draw [color=black!20] (88,-2)--(88,-6);

\node at (15,10.5) {\bf 1};
\node at (51,10.5) {\bf 1};

\node at (-21,10.5) {\bf 2};
\node at (87,10.5) {\bf 2};

\node at (33,16.5) {\bf 3};
\node at (33,4.5) {\bf 3};

\node at (69,16.5) {\bf 4};
\node at (69,4.5) {\bf 4};
\node at (-3,16.5) {\bf 4};
\node at (-3,4.5) {\bf 4};

\node at (15,22.5) {\bf 5};
\node at (51,22.5) {\bf 5};
\node at (15,-1.5) {\bf 5};
\node at (51,-1.5) {\bf 5};

\end{tikzpicture}
\end{center}
\caption{Tiling Pattern.}\label{fig_pattern}
\end{figure}
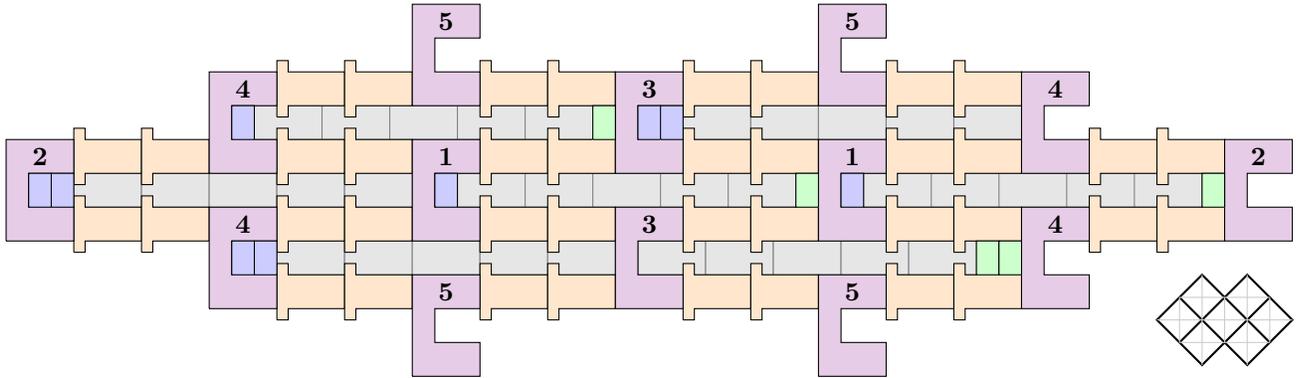

We begin with a pair of horizontally neighboring connectors that are $2t+1$ segments apart (labeled $1$ in Figure~\ref{fig_pattern}). By the distance condition, on the same horizontal row, there must be another two connectors $2t+1$ segment away to the left and the right (labeled $2$). On the horizontal row above (and below) the first row, we must place a connector (labeled $3$) exactly in the middle between the first two connectors (labeled $1$). By repeatedly applying the positional properties of the connectors, we continue to place the connectors labeled $4$ and $5$ and so on. This process of placing the connectors continues indefinitely to the entire plane.  Therefore, without loss of generality, if we place one of the connectors at the origin of the plane, then the set of connectors in any tiling forms a rigid lattice structure.

\item (\textbf{Filling the gaps.}) By the previous steps, we have a rigid lattice of connectors. Between any two horizontally neighboring connectors at the same row, there is a flexibility of placing an encoder from one of the $n$ different configurations, which simulates one of the $n$ Wang tiles. To tile the entire plane, it depends on whether the remaining gaps can be filled completely. The majority of the gaps are between the encoders of different rows. Because of the building blocks $Y^+$ and $Y^-$ of the encoders, these gaps can only be filled by the linkers. As mentioned before, the linkers always link a building block $l$ of one encoder with another $l$ of another encoder, or link a building block $r$ with another $r$. Therefore, the gaps between different rows of encoders can be filled by linkers if and only if the simulated Wang tiles are matched in colors for any pair of adjacent simulated edges of two encoders. Note that there are still some tiny gaps (the gaps in the building blocks $l$ and $r$ of the encoders that are not matched by the linkers) after placing the linkers, and those tiny gaps can always be filled by the tiny $T$-fillers. 
\end{itemize}

By the above arguments, the set of $7$ polyominoes can tile the plane if and only if the set of Wang tiles that are simulated by the $7$ polyominoes can tile the plane. We have reduced the undecidable Wang's domino problem \cite{b66} to the problem of tiling the plane with a set of $7$ polyominoes. This completes the proof.
\end{proof}

Because there exist aperiodic sets of Wang tiles (there are quite a few of them, see \cite{b66,jr21,kari96,r71}), so for each aperiodic set $W$ of Wang tiles, by the proof of Theorem \ref{thm_main}, the set $P$ of $7$ polyominoes constructed to simulate $W$ is also aperiodic. So Corollary \ref{cor_main} is a direct consequence of Theorem \ref{thm_main}.

\noindent \textbf{Remark 1.} As mentioned in Section \ref{sec_intro}, the most crucial novel technique in the proof of Theorem \ref{thm_main} is the interlacing method for simulating a set of Wang tiles in a single polyomino. However, to establish a complete system (i.e., the set of $7$ polyominoes) that simulates a set of Wang tiles perfectly, several features are also indispensable. For example, the length of the linkers and the connectors plays an important role in enforcing the rigid structure of the connectors.

\section{Conclusion}\label{sec_conclu}

We prove the undecidability of translational tiling of the plane with a set of $7$ polyominoes in this paper. This extends our knowledge of the undecidability of the more general problem of translational tiling of $\mathbb{Z}^n$ with a set of $k$ tiles (both $n$ and $k$ are fixed).

\begin{Problem}[Translational tiling of $\mathbb{Z}^n$ with a set of $k$ tiles] \label{pro_gen}
A tile is a finite subset of $\mathbb{Z}^n$. Let $k$ and $n$ be fixed positive integers. Given a set $S$ of $k$ tiles in $\mathbb{Z}^n$, is there an algorithm to decide whether $\mathbb{Z}^n$ can be tiled by translated copies of tiles in $S$?
\end{Problem}

With our results in this paper, the current knowledge about the decidability and undecidability of Problem \ref{pro_gen} can be summarized in Figure \ref{fig_nk}. The green region is known to be decidable \cite{bn91,b20, gt21, s93, w15}, the red region is known to be undecidable \cite{o09,yz24,yz24c,yz24d}, and the yellow region is possibly undecidable \cite{gt24b}. It remains open to find a fixed dimension $n$ that translational tiling of $\mathbb{Z}^n$ with one tile or two tiles is undecidable. The boundary of the yellow region in Figure \ref{fig_nk} is to demonstrate the idea that as the dimension~$n$ increases, it may need fewer tiles to get undecidable results for translational tiling of $\mathbb{Z}^n$. For the plane, the decidability or undecidability of translational tiling of $\mathbb{Z}^2$ with a set of $k$ ($2\leq k\leq 6$) polyominoes is still open.

As a consequence of our undecidability results, we also prove the existence of aperiodic sets of $7$ polyominoes for translational tiling of the plane. This breaks the previous 30-year-old record of Ammann on the minimum size of aperiodic sets for translational tiling of the plane. Our aperiodic sets of $7$ polyominoes are constructed from simulating aperiodic sets of Wang tiles \cite{b66,jr21,kari96,r71}, so they are much more complicated than Ammann's set of $8$ tiles \cite{ags92}. A natural problem for future study is to find simpler aperiodic sets of size $7$ or less for the translational tiling of the plane.


\begin{figure}[H]
\begin{center}
\begin{tikzpicture}[scale=1]

\draw [fill=green!20,dashed] (11.9,0)--(0,0)--(0,2)--(1,2)--(1,1)--(11.9,1);

\begin{scope}
    \clip (0,9) rectangle (2,10.9);
\filldraw [yellow!50] (0,11)--(0,8)--(3,8)--(3,11); 
\end{scope}

\begin{scope}
    \clip (2,1) rectangle (11.9,10.9);
\draw [fill=red!20,dashed] (12,1)--(6,1)--(6,2)--(3,2)--(3,3)--(2,3)--(2,4)--(2,12)--(12,12); 
\end{scope}

\begin{scope}
    \clip (0.64,3.9) rectangle (2,10.9);

\filldraw [yellow!50,dashed] plot [smooth] coordinates {(-1,11) (1,8.5) (1.5,6.5) (2,4)  (2.5,1) (9,0) (12,1) (12,15) (13,12)};
\end{scope}

\draw[help lines, color=gray, dashed]  (-0.1,-0.1) grid (11.9,10.9);
\draw[->,ultra thick] (-0.5,0)--(12,0) node[right]{$k$};
\draw[->,ultra thick] (0,-0.5)--(0,11) node[above]{$n$};

\foreach \y in {1,2,3,4,5}
{
\node at (-0.5,\y-0.5) {\y};
}
\node at (-0.9,9.7) {some};
\node at (-0.8,9.3) {large $n$?};

\foreach \x in {1,...,11}
{
\node at (\x-0.5,-0.5) {\x};
}

\end{tikzpicture}
\end{center}
\caption{Translational tiling problem of $\mathbb{Z}^n$ with a set of $k$ tiles.}\label{fig_nk}
\end{figure}
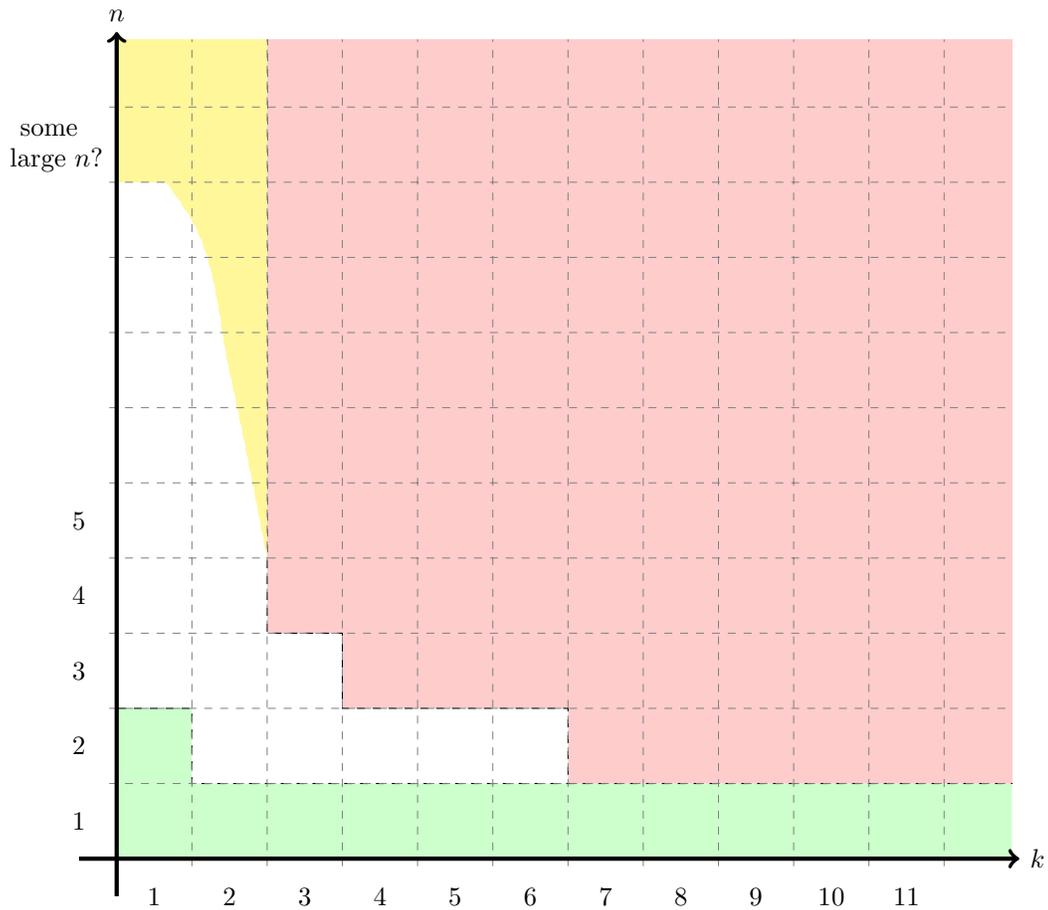

\section*{Acknowledgements}
The first author was supported by the Research Fund of Guangdong University of Foreign Studies (Nos. 297-ZW200011 and 297-ZW230018), and the National Natural Science Foundation of China (No. 61976104).



\end{document}